\documentclass[11pt]{article}
\usepackage{amscd}
\usepackage{amsfonts}
\usepackage{amsmath}
\usepackage{amssymb}
\usepackage{amsthm}
\usepackage{bbm}
\usepackage{CJK}
\usepackage{fancyhdr}
\usepackage{graphicx}
\usepackage{indentfirst}
\usepackage{latexsym}
\usepackage{mathrsfs}
\usepackage{xypic}
\usepackage[colorlinks,linkcolor=red, anchorcolor=blue]{hyperref}

\newtheorem{theorem}{Theorem}[section]
\newtheorem{lemma}[theorem]{Lemma}
\newtheorem{definition}[theorem]{Definition}
\newtheorem{proposition}[theorem]{Proposition}
\newtheorem{example}[theorem]{Example}

\newtheorem{corollary}[theorem]{Corollary}
\newtheorem{remark}[theorem]{Remark}

\usepackage[top=1in,bottom=1in,left=1.25in,right=1.25in]{geometry}
\def\<{\langle}
\def\>{\rangle}
\def\a{\alpha}
\def\b{\beta}

\def\g{\gamma}

\def\lr{\longrightarrow}

\def\o{\otimes}

\def\r{\rho}

\def\th{\theta}

\def\la{\lambda}

\date{}
\begin{document}
\renewcommand{\baselinestretch}{1.2}
\renewcommand{\arraystretch}{1.0}
\title{\bf The braided monoidal structure on the category of Hom-type Doi-Hopf modules}
\author{{\bf Daowei Lu\footnote {Corresponding author},
\ Shuanhong Wang}
{\small }}
 \maketitle

\begin{center}
\begin{minipage}{12.cm}

\noindent{\bf Abstract.} Let $(H,\a_H)$ be a Hom-Hopf algebra, $(A,\a_A)$ a right $H$-comodule algebra and $(C,\a_C)$ a left $H$-module coalgebra. Then we have the category $_A\mathcal{M}(H)^C$ of Hom-type Doi-Hopf modules. The aim of this paper is to make the category $_A\mathcal{M}(H)^C$ into a braided monoidal category. Our construction unifies quasitriangular and coquasitriangular Hom-Hopf algebras and Hom-Yetter-Drinfeld modules. We study tensor identities for monoidal categories of Hom-type Doi-Hopf modules.  Finally we show that the category $_A\mathcal{M}(H)^C$ is isomorphic to $A\#C^*$-module category.
 \\

\\

 \end{minipage}
 \end{center}
 \normalsize\vskip1cm

\section*{Introduction}

The Doi-Hopf datum $(H,A,C)$ introduced in \cite{D} consists of a Hopf algebra $H$, a right $H$-comodule algebra $A$ and a left $H$-module coalgebra $C$. The Doi-Hopf module $M$ over $(H,A,C)$ is both a left $A$-module and a right $C$-comodule satisfying certain compatible condition. The category of Doi-Hopf modules over $(H,A,C)$ is denoted by $_A\mathcal{M}(H)^C$.  The research of $_A\mathcal{M}(H)^C$ turn out to be very essential: it is pointed out in \cite{D} that categories such as module and comodule over bialgebra, in \cite{S} that the Hopf modules category, and in \cite{CMZ, RT} that the Yetter-Drinfeld modules category are special cases of $_A\mathcal{M}(H)^C$. For a further study of Doi-Hopf modules, we refer to \cite{CVZ}.

Braided monoidal categories give rise to solutions to the Quantum Yang-Baxter equations. The classical braided monoidal categories come from the representations of quasitriangular Hopf algebras which have been very widely studied.
One of the most important examples of quasitriangular Hopf algebras is the Drinfeld double of any finite dimensional Hopf algebra. As such, they are interesting to different research communities in mathematical physics (see \cite{M}  for example).

 From a physical viewpoint, Hom-type structures are important because they are related
to vertex operator algebras and string theory(see \cite{AGS},\cite{AS}). Hom-Lie algebras were introduced in \cite{HLS} to describe the structures on some $q$-deformations of the Witt and the Virasoro algebras, both of which are important in vertex operator algebras and string theory. Also in this paper Hom-type algebras have been introduced in the form of Hom-Lie algebras, where the Jacobi identity was twisted along a linear endomorphism. Meanwhile, Hom-associative algebras have been suggested in \cite{MS1} to give rise to a Hom-Lie algebra using the commutator bracket. Other Hom-type structures such as Hom-coalgebras, Hom-bialgebras, Hom-Hopf algebras as well as their properties have been considered in \cite{MS2}.

Yang-Baxter equations and its solutions have been generalized to Hom-bialgebra in \cite{Y}, and the author pointed out that a quasitriangular Hom-bialgebra $(H,R)$ provides solutions of a class of Hom Yang-Baxter equations:
\begin{eqnarray*}
&&(R_{12}R_{13})R_{23}=R_{23}(R_{13}R_{12}),\\
&&R_{12}(R_{13}R_{23})=(R_{23}R_{13})R_{12}.
\end{eqnarray*}

Motivated by these ideas, in this paper our aim is to make the Hom-type Doi-Hopf module category $_A\mathcal{M}(H)^C$ into a braided monoidal category. We will prove that $_A\mathcal{M}(H)^C$ and $_{A\#C^*}\mathcal{M}$ are isomorphic as braided monoidal category. Hence $A\#C^*$, as a generalization of the Drinfeld double $D(H)$, is quasitriangular.

This article is organized as follows:

In section 1, we will recall the basic definitions and results on Hom-Hopf algebra, including Hom-Hopf algebra, Hom-module coalgebra, Hom-comoudle algebra and quasitriangular Hom-bialgebra.

In section 2, we will discuss when the category $_A\mathcal{M}(H)^C$ becomes monoidal (see Proposition 2.3). Then Yetter-Drinfeld category $_H\mathcal{YD}^H$ is isomorphic to $_H\mathcal{M}(H^{op}\o H)^H$ as monoidal category. Maps between the underlying Hom-Hopf algebras, Hom-algebras and Hom-coalgebras give rise to functors between Hom-type Doi-Hopf modules (see Proposition 2.6 and 2,7).
We will also discuss particular situations which give rise to the tensor identities (see Theorem 2.8 and 2.10).

In section 3, we will present the necessary and sufficient conditions for $_A\mathcal{M}(H)^C$ to be braided. We point out that this comes down to a convolution invertible map $Q:C\o C\rightarrow A\o A$ subject to certain compatible relations (see Theorem 3.8). Especially, on one hand when $C=k$ this conditions turn out to be equivalent to $(A,Q(1))$ being quasitriangular. On the other hand when $A=k$ then we recover the definition of a coquasitriangular Hom-Hopf algebra.

In section 4, we will prove the monoidal category isomorphism $_A\mathcal{M}(H)^C\simeq\ _{A\#C^*}\mathcal{M}$ (see Proposition 4.4). Since $_A\mathcal{M}(H)^C$ is braided, so is $_{A\#C^*}\mathcal{M}$. In the case of $_H\mathcal{M}(H^{op}\o H)^H$, $A\#C^*$ reduces to the Drinfeld double $D(H)=H\bowtie H^*$.

Throughout this article, all the vector spaces, tensor product and homomorphisms are over a fixed field $k$.  For a coalgebra $C$, we will use the Heyneman-Sweedler's notation $\Delta(c)=c_{1}\otimes c_{2},$
for all $c\in C$ (summation omitted).

\section{Preliminary}

\def\theequation{1.\arabic{equation}}
\setcounter{equation} {0} \hskip\parindent

In this section, we will recall from \cite{MP,MS2} the basic definitions and results on the  Hom-Hopf algebras,  Hom-modules and  Hom-comodules.

 A unital Hom-associative algebra is a triple $(A,m,\alpha)$ where $\alpha:A\lr A$ and $m:A\o A\lr A$ are linear maps, with notation $m(a\o b)=ab$ such that for any $a,b,c\in A$,
 \begin{eqnarray*}
&&\alpha(ab)=\alpha(a)\alpha(b),\ \alpha(1_{A})=1_{A},\\
&&1_{A}a=\alpha(a)=a1_{A},\ \alpha(a)(bc)=(ab)\alpha(c).
\end{eqnarray*}
A linear map $f:(A,\mu_{A},\alpha_{A})\lr (B,\mu_{B},\alpha_{B})$ is called a morphism of Hom-associative algebra if $\alpha_{B}\circ f=f\circ\alpha_{A}$, $f(1_{A})=1_{B}$ and $f\circ\mu_{A}=\mu_{B}\circ(f\o f).$

 A counital Hom-coassociative coalgebra is a triple $(C,\Delta,\varepsilon,\alpha)$ where $\alpha:C\lr C$, $\varepsilon:C\lr k$, and $\Delta:C\lr C\o C$ are linear maps such that
 \begin{eqnarray*}
&&\varepsilon\circ\alpha=\varepsilon,\ (\alpha\o\alpha)\circ\Delta=\Delta\circ\alpha,\\
&&(\varepsilon\o id)\circ\Delta=\alpha=(id\o\varepsilon)\circ\Delta,\\
&&(\Delta\o\alpha)\circ\Delta=(\alpha\o\Delta)\circ\Delta.
\end{eqnarray*}

A linear map $f:(C,\Delta_{C},\alpha_{C})\lr (D,\Delta_{D},\alpha_{D})$ is called a morphism of Hom-coassociative coalgebra if $\alpha_{D}\circ f=f\circ\alpha_{C}$, $\varepsilon_{D}\circ f=\varepsilon_{C}$ and $\Delta_{D}\circ f=(f\o f)\circ\Delta_{C}.$

In what follows, we will always assume all Hom-algebras are unital and Hom-coalgebras are counital.

A Hom-bialgebra is a quadruple $(H,\mu,\Delta,\alpha)$, where $(H,\mu,\alpha)$ is a Hom-associative algebra and $(H,\Delta,\alpha)$ is a Hom-coassociative coalgebra such that $\Delta$ and $\varepsilon$ are morphisms of Hom-associative algebra.

A Hom-Hopf algebra $(H,\mu,\Delta,\alpha)$ is a Hom-bialgebra $H$ with a linear map $S:H\lr H$(called antipode) such that
\begin{eqnarray*}
&&S\circ\alpha=\alpha\circ S,\\
&&S(h_{1})h_{2}=  h_{1}S(h_{2})=\varepsilon(h)1,
\end{eqnarray*}
for any $h\in H$. For $S$ we have the following properties:
 \begin{eqnarray*}
&&S(h)_{1}\o S(h)_{2}=  S(h_{2})\o S(h_{1}),\\
&&S(gh)=S(h)S(g),\ \varepsilon\circ S=\varepsilon.
 \end{eqnarray*}
For any Hopf algebra $H$ and any Hopf algebra endomorphism $\alpha$ of $H$, there exists a Hom-Hopf algebra $H_{\alpha}=(H,\alpha\circ\mu,1_{H},\Delta\circ\alpha,\varepsilon,S,\alpha)$.

Let $(A,\alpha_{A})$ be a Hom-associative algebra, $M$ a linear space and $\alpha_{M}:M\lr M$ a linear map. A left $A$-module structure on $(M,\alpha_{M})$ consists of a linear map $A\o M\lr M$, $a\o m\mapsto a\cdot m$, such that
 \begin{eqnarray*}
&&1_{A}\cdot m=\alpha_{M}(m),\\
&&\alpha_{M}(a\cdot m)=\alpha_{A}(a)\cdot\alpha_{M}(m),\\
&&\alpha_{A}(a)\cdot(b\cdot m)=(ab)\cdot\alpha_{M}(m),
\end{eqnarray*}
for any $a,b\in A$ and $m\in M.$

Similarly we can define the right $A$-modules. Let $(M,\mu)$ and $(M',\mu')$ be two left $A$-modules, then a linear map $f:M\rightarrow N$ is a called left $A$-module map if $f(a\cdot m)=aa\cdot f(m)$ for any $a\in A$, $m\in M$ and $f\circ\mu=\mu'\circ f$.

Suppose that $(M,\mu)$ is a right $A$-module and $(N,\nu)$ is a left $A$-module, then denote by $X$ the subspace of $M\o N$ spanned by the elements of the type $\Big\{\sum m_i\cdot a\o \nu(n_i)-\sum \mu(m_i)\o a\cdot n_i|\forall m_i\in M,n_i\in N,a\in A\Big\}$. Then we have the quotient space
$$M\o_AN=(M\o N)/X.$$

Let $(C,\alpha_{C})$ be a Hom-coassociative coalgebra, $M$ a linear space and $\alpha_{M}:M\lr M$ a linear map. A right $C$-comodule structure on $(M,\alpha_{M})$ consists a linear map $\rho:M\lr M\o C$ such that
 \begin{eqnarray*}
&&(id\o\varepsilon_{C})\circ\rho=\alpha_{M},\\
&&(\alpha_{M}\o\alpha_{C})\circ\rho=\rho\circ\alpha_{M},\\
&&(\rho\o\alpha_{C})\circ\rho=(\alpha_{M}\o\Delta)\circ\rho.
 \end{eqnarray*}
Let $(M,\mu)$ and $(M',\mu')$ be two right $(C,\gamma)$-comodules, then a linear map $g:M\lr M'$ is a called right $C$-comodule map if $g\circ \mu=\mu'\circ g$ and $\rho_{M'}\circ g=(g\otimes id)\circ\rho_{M}$.

Suppose that $(M,\mu)$ is a right $C$-comodule and $(N,\nu)$ is a left $C$-comodule, then we have the cotensor product
$$M\square_C N=\Big\{\sum m_i\o n_i\in M\o N|\sum\mu(m_i)\o n_{i(-1)}\o n_{i(0)}=\sum m_{i(0)}\o m_{i(1)}\o\nu(n_i)\Big\}.$$

Let $(H,\a_H)$ be a Hom-bialgebra. A Hom-algebra $(A,\a_A)$ is called a right $H$-comodule algebra if $A$ is a right $H$-comodule via $\r$, and
$$\r(ab)=\r(a)\r(b),~~~\r(1_A)=1_A\o1_H,$$
for all $a,b\in A$. A Hom-coalgebra $(C,\a_C)$ is called a left $H$-module coalgebra if $C$ is a left $H$-module and
$$\Delta_C(h\cdot c)=\Delta_H(h)\cdot\Delta_C(c),~~~\varepsilon_C(h\cdot c)=\varepsilon_H(h)\varepsilon_C(c),$$
for all $h\in H,c\in C$.

Suppose that $(A,\a_A)$ is a right $H$-comodule algebra and $(C,\a_C)$ is a left $H$-module coalgebra. An object $(M,\a_M)$ is called a Doi-Hopf module over $(H,A,C)$ if $M$ is a left $A$-module via $\cdot$ and a right $C$-comodule via $\r'$ satisfying
$$\r(a\cdot m)=a_{(0)}\cdot m_{(0)}\o a_{(1)}\cdot m_{(1)},$$
for all $a\in A$ and $m\in M$.

The category of Doi-Hopf modules over $(H,A,C)$ will be denoted by $_A\mathcal{M}(H)^C$.

Recall from \cite{K} that a monoidal category $(\mathcal{C},\o,I,a,l,r)$ is braided if there exists a family of natural isomorphisms $t=\big\{t_{_{V,W}}:V\o W\rightarrow W\o V\big\}_{V,W\in\mathcal{C}}$ such that
\begin{eqnarray}
&&a_{_{V,W,U}}t_{_{U,V\o W}}a_{_{U,V,W}}=(id_V\o t_{_{U,W}})a_{_{V,U,W}}(t_{_{U,V}}\o id_W),\\
&&a^{-1}_{_{W,U,V}}t_{_{U\o V,W}}a^{-1}_{_{U,V,W}}=(t_{_{U,W}}\o id_V)a^{-1}_{_{U,W,V}}(id_U\o t_{_{V,W}}),
\end{eqnarray}
for all objects $U,V,W$ in $\mathcal{C}$.

Hom-bialgebra $(H,\a_H)$ is quasitriangular \cite{Y} if there exists an element $R=R^1\o R^2\in H\o H$ such that
\begin{eqnarray*}
&&\Delta^{op}(h)R=R\Delta(h),\\
&&(\Delta\o\a_H)R=R^{13}R^{23},\\
&&(\a_H\o\Delta)R=R^{13}R^{12},
\end{eqnarray*}
where $R^{13}=R^1\o1\o R^2$, $R^{12}=R^1\o R^2\o1$ and $R^{23}=1\o R^1\o R^2$.

\section{Monoidal structure on the category $_A\mathcal{M}(H)^C$}
\def\theequation{2.\arabic{equation}}
\setcounter{equation} {0} \hskip\parindent

In this section, we will discuss when the category $_A\mathcal{M}(H)^C$ becomes monoidal.
In what follows, let $(H,\a_H)$ be a Hom-Hopf algebra, $(A,\a_A)$ a right $H$-comodule algebra and $(C,\a_C)$ a left $H$-module coalgebra. First of all, we have the following definition.

\begin{definition}
The Doi-Hopf datum $(H,A,C)$ called is a monoidal Doi-Hopf datum if $(A,\a_A)$ and $(C,\a_C)$ are Hom-bialgebras with the following compatible conditions:
\begin{eqnarray}
&&a_{(0)1}\o a_{(0)2}\o\a^2_H(a_{(1)})\cdot(cd)=a_{1(0)}\o a_{2(0)}\o(a_{1(1)}\cdot c)(a_{2(1)}\cdot d),\label{A}\\
&&\varepsilon(a)1_C=\varepsilon(a_{(0)})a_{(1)}\cdot 1_C.
\end{eqnarray}
\end{definition}

\begin{remark}
If $(H,A,C)$ is a monoidal Doi-Hopf datum, then $A$ and $C$ are also Doi-Hopf modules. The left $A$-action and the right $C$-coaction on $C$ are given by
\begin{eqnarray*}
&&a\cdot c=\varepsilon(a_{(0)})a_{(1)}\cdot c,~~~\r(c)=c_1\o\a^{-1}_C(c_2).
\end{eqnarray*}
The left $A$-action and the right $C$-coaction on $A$ are given by
\begin{eqnarray*}
&&a\cdot b=\a^{-1}_A(a)b,~~~\r(a)=a_{(0)}\o a_{(1)}\cdot 1_C.
\end{eqnarray*}
\end{remark}

\begin{proposition}
Let $(H,A,C)$ be a monoidal Doi-Hopf datum. For all Doi-Hopf modules $(M,\a_M)$ and $(N,\a_N)$, $(M\o N,\a_M\o\a_N)$ still makes a Doi-Hopf module with the following structures:
\begin{eqnarray}
&&a\cdot(m\o n)=a_1\cdot m\o a_2\cdot n,\\
&&\r(m\o n)=m_{(0)}\o n_{(0)}\o\a^{-2}_C(m_{(1)}n_{(1)}),
\end{eqnarray}
for all $a\in A,m\in M$ and $n\in N$. The category $\mathcal{C}=\ _A\mathcal{M}(H)^C$ is a monoidal category.
\end{proposition}

\begin{proof}
We need only to verify the compatible condition. Indeed for $a\in A,m\in M$ and $n\in N$,
$$\begin{aligned}
\r(a\cdot(m\o n))&=a_{1(0)}\cdot m_{(0)}\o a_{2(0)}\cdot n_{(0)}\o\a^{-2}_C((a_{1(1)}\cdot m_{(1)})(a_{2(1)}\cdot n_{(1)}))\\
                 &\stackrel{(\ref{A})}=a_{(0)1}\cdot m_{(0)}\o a_{(0)2}\cdot n_{(0)}\o a_{(1)}\cdot \a^{-2}_C(m_{(1)}n_{(1)})\\
                 &=a_{(0)}\cdot(m_{(0)}\o n_{(0)})\o a_{(1)}\cdot \a^{-2}_C(m_{(1)}n_{(1)}).
\end{aligned}$$

It is now obvious that tensor product defines a functor $\mathcal{C}\times\mathcal{C}\rightarrow \mathcal{C}$. Under the trivial $A$-action and $C$-coaction, $k$ is a Doi-Hopf module, and $k$ is the unit object in $\mathcal{C}$. Let $M,N,P$ be Doi-Hopf modules, from \cite{MP} the isomorphisms
\begin{eqnarray*}
&&a_{_{M,N,P}}:(M\hat{\o}N)\hat{\o}P\rightarrow M\hat{\o}(N\hat{\o}P),~~~ (m\o n)\o p\mapsto \a^{-1}_M(m)\o(n\o\a_N(p)),\\
&&l_M:k\hat{\o}M\rightarrow M,~~~\la\o m\mapsto \la\a^{-1}_M(m),\\
&&r_M:M\hat{\o}k\rightarrow M,~~~m\o\la \mapsto \la\a^{-1}_M(m),
\end{eqnarray*}
define respectively the associator, left and right unit of $\mathcal{C}$.
\end{proof}

\begin{example}
 Suppose that $C$ is a left $H$-module bialgebra, which means that $C$ as a Hom-bialgebra is not only a left $H$-module algebra but also a left $H$-comodule algebra. Let $A$ be a Hom-bialgebra and a right $H$-comodule algebra such that for all $a\in A$,
$$a_{(0)1}\o a_{(0)2}\o a_{(1)1}\o a_{(1)2}=a_{1(0)}\o a_{2(0)}\o a_{1(1)}\o a_{2(1)}.$$
Clearly $(H,A,C)$ is a monoidal Doi-Hopf datum.

Actually $(H,\a_H)$ itself is a left $H$-module bialgebra, with the left adjoint action $h\cdot g=(h_1\a^{-1}_H(g))S\a_H(h_2)$ and the coaction afforded by the comultiplication.

Let $(H,\a_H)$ be a cocommutitive Hom-Hopf algebra, and $(C,\a_C)$ be a left $H$-module bialgebra. Easy to see that $(H,H,C)$ is a monoidal Doi-Hopf datum where $H$ is considered as a right $H$-comodule via comultiplication.
\end{example}

\begin{proposition}
Let $(H,\a_H)$ be a Hom-Hopf algebra with bijective antipode, then $H$ is a right $H^{op}\o H$-comodule with the following structure
$$\r(h)=\a^{-1}_H(h_{12})\o S^{-1}\a^{-2}_H(h_{11})\o\a^{-1}_H(h_2),$$
for all $h\in H$.
Define the left $H^{op}\o H$-action on $H$ by $(h\o g)\cdot l=(g\a^{-1}_H(l))\a_H(h)$, then $(H^{op}\o H,H,H)$ is a Doi-Hopf datum. Moreover we have the category isomorphism
$$_H\mathcal{M}(H^{op}\o H)^H\simeq\ _H\mathcal{YD}^H.$$
\end{proposition}

\begin{proof}
Apparently $H$ is a Doi-Hopf module over the datum $(H^{op}\o H,H,H)$. Next we will show that $_H\mathcal{M}(H^{op}\o H)^H$ is a monoidal category. Take $A=H$ and $C=H^{op}$ as Hom-bialgebra, for all $h,g,g'\in H$,
$$\begin{aligned}
&h_{(0)1}\o h_{(0)2}\o\a_H^2(h_{1})\cdot(g\bullet g')\\
=&\a_H^{-1}(h_{121})\o\a_H^{-1}(h_{122})\o(\a_H(h_2)\a_H^{-1}(g'g))S^{-1}\a_H(h_{11})\\
=&h_{12}\o h_{21}\o(h_{22}\a_H^{-1}(g'g))S^{-1}\a_H(h_{11})\\
=&h_{12}\o h_{21}\o(h_{22}g')(gS^{-1}(h_{11}))\\
=&\a_H^{-1}(h_{112})\o\a_H^{-1}(h_{212})\o[\a_H^{-1}(h_{22}g')(S^{-1}(\a_H^{-2}(h_{211}))\a_H^{-1}(h_{12}))](gS^{-1}\a_H^{-1}(h_{111}))\\
=&\a_H^{-1}(h_{112})\o\a_H^{-1}(h_{212})\o[\a_H^{-1}(h_{22}g')S^{-1}(\a_H^{-1}(h_{211}))][h_{12}(\a_H^{-1}(g)S^{-1}\a_H^{-2}(h_{111}))]\\
=&\a_H^{-1}(h_{112})\o\a_H^{-1}(h_{212})\\
&\o[(S^{-1}\a_H^{-2}(h_{111})\o\a_H^{-1}(h_{12}))\cdot g]\bullet[(S^{-1}(\a_H^{-2}(h_{211})\o\a_H^{-1}(h_{22})))\cdot g']\\
=&h_{1(0)}\o h_{2(0)}\o (h_{1(1)}\cdot g)\bullet(h_{1(2)}\cdot g'),
\end{aligned}$$
and
$$\varepsilon(h_{(0)})h_{(1)}\cdot 1=h_2S^{-1}(h_{1})=\varepsilon(h)1.$$

For any object $(M,\a_M)$ in $_H\mathcal{YD}^H$ and for all $h\in H,m\in M$,
$$\begin{aligned}
\r(h\cdot m)&=\a^{-1}_H(h_{21})\cdot m_{(0)}\o[\a^{-2}_H(h_{22})\a^{-1}_H(m_{(1)})]S^{-1}(h_1)\\
            &=\a^{-1}_H(h_{12})\cdot m_{(0)}\o[\a^{-1}_H(h_{2})\a^{-1}_H(m_{(1)})]S^{-1}\a^{-1}_H(h_{11})\\
            &=h_{(0)}\cdot m_{(0)}\o h_{(1)}\cdot m_{(1)}.
\end{aligned}$$
That is, $(M,\a_M)$ is an object in $_H\mathcal{M}(H^{op}\o H)^H$. Conversely for any object $(M,\a_M)$ in $_H\mathcal{M}(H^{op}\o H)^H$, by the same computation it is also an object in $_H\mathcal{YD}^H$. Thus these two categories are isomorphic and have the same monoidal structure.

The proof is completed.
\end{proof}

Let $(H,A,C)$ and $(H',A',C')$ be two Doi-Hopf data. A morphism $\varphi:(H,A,C)\rightarrow(H',A',C')$ consists of three maps $\theta:H\rightarrow H',$ $\b:A\rightarrow A'$ and $\g:C\rightarrow C'$ which are, respectively a Hom-Hopf algebra map, a Hom-algebra map and a Hom-coalgebra map satisfying
\begin{eqnarray}
&&\g(h\cdot c)=\theta(h)\cdot\g(c),\\
&&\r(\b(a))=\b(a_{(0)})\o\th(a_{(1)})
\end{eqnarray}
for all $h\in H,c\in C$ and $a\in A$.

\begin{proposition}
The functor $F:\ _A\mathcal{M}(H)^C\longrightarrow\ _{A'}\mathcal{M}(H)^{C'}$ given by
$$F(M)=A'\o_A M,$$
where the Doi-Hopf module structure on $A'\o_A M$ is defined in the following way
\begin{eqnarray}
&&a'\cdot(b'\o m)=\a_{A'}(a')b'\o\a_M(m),\\
&&\r(b'\o m)=b'_{(0)}\o m_{(0)}\o\a^{-2}_{M'}(b'_{(1)}\cdot \g(m_{(1)})),
\end{eqnarray}
for all $a',b'\in A'$ and $m\in M$.
\end{proposition}

\begin{proof}
Clearly for any object $(M,\a_M)$ in $_A\mathcal{M}(H)^C$, $A'\o_A M$ is a left $A'$-module and right $C'$-comodule. Then for all $a',b'\in A'$ and $m\in M$,
$$\begin{aligned}
\r(a'\cdot (b'\o m))&=\r(\a_{A'}(a')b'\o\a_M(m))\\
                    &=\a_{A'}(a'_{(0)})b'_{(0)}\o\a_M(m_{(0)})\o\a^{-2}_{C'}(\a_{A'}(a'_{(1)})b'_{(1)}\cdot\g(\a_M(m_{(1)})))\\
                    &=a'_{(0)}\cdot(b'_{(0)}\o m_{(0)})\o a'_{(1)}\cdot\a^{-2}_{C'}(b'_{(1)}\cdot\g(m_{(1)}))\\
                    &=a'_{(0)}\cdot(b'\o m)_{(0)}\o a'_{(1)}\cdot(b'\o m)_{(1)}.
\end{aligned}$$

This completes the proof.

\end{proof}

\begin{proposition}
The functor $G:\ _{A'}\mathcal{M}(H)^{C'}\longrightarrow\ _A\mathcal{M}(H)^C$ given by
$$G(N')=N'\square_{C'}C,$$
where the Doi-Hopf module structure on $N'\square_{C'}C$ is defined in the following way
\begin{eqnarray}
&&a\cdot(n'\o c)=\b(a_{(0)})\cdot n'\o a_{(1)}\cdot c,\\
&&\r(n'\o c)=\a_{N'}(n')\o c_1\o\a^{-1}_C(c_2),
\end{eqnarray}
for all $a\in A,c\in C$ and $n'\in N'$.

Moreover $(F,G)$ is a pair of adjoint functors.
\end{proposition}

\begin{proof}
First of all, we need to show the action and coaction on $G(N')$ is well defined for any object $N'\in\ _{A'}\mathcal{M}(H)^{C'}$. For all $a\in A,c\in C$ and $n'\in N'$,
$$\begin{aligned}
&(\b(a_{(0)})\cdot n')_{(0)}\o(\b(a_{(0)})\cdot n')_{(1)}\o\a_H(a_{(1)})\cdot \a_C(c)\\
&=\b(a_{(0)(0)})\cdot n'_{(0)}\o a_{(0)(1)}\cdot n'_{(1)}\o\a_H(a_{(1)})\cdot \a_C(c)\\
&=\a_{A'}(\b(a_{(0)}))\cdot \a_{N'}(n')\o a_{(1)1}\cdot \g(c_1)\o a_{(1)2}\cdot c_2,
\end{aligned}$$
thus the action is well defined. And
$$\begin{aligned}
&\a_{N'}(n')_{(0)}\o\a_{N'}(n')_{(1)}\o \a_C(c_1)\o\a^{-1}_C(c_2)\\
&=\a_{N'}(n'_{(0)})\o\a_{C'}(n'_{(1)})\o \a_C(c_1)\o\a^{-1}_C(c_2)\\
&=\a^2_{N'}(n')\o\a_{C'}(\g(c_1))\o c_{21}\o\a^{-2}_C(c_{22})\\
&=\a^2_{N'}(n')\o\g(c_{11})\o c_{12}\o\a^{-1}_C(c_{2}),
\end{aligned}$$
thus the coaction is well defined. It is straightforward to verify that $G(N')$ is a left $A$-module and right $C$-comodule. For the compatibility we have
$$\begin{aligned}
\r(a\cdot(n'\o c))&=\a_{A'}(\b(a_{(0)}))\cdot \a_{N'}(n')\o a_{(1)1}\cdot c_1\o \a^{-1}_H(a_{(1)2})\cdot \a^{-1}_C(c_2)\\
                  &=\b(a_{(0)(0)})\cdot \a_{N'}(n')\o a_{(0)(1)}\cdot c_1\o a_{(1)}\cdot \a^{-1}_C(c_2)\\
                  &=a_{(0)}\cdot(\a_{N'}(n')\o c_1)\o a_{(1)}\cdot\a^{-1}_C(c_2).
\end{aligned}$$
Then $G(N')$ is an object in $_A\mathcal{M}(H)^C$. Finally we show $(F,G)$ is a pair of adjoint functors. Define the linear maps
$\eta_M:M\longrightarrow GF(M)$ by
$$m\mapsto (1\o \a_M(m_{(0)}))\o m_{(1)},$$
and
$\delta_{M'}:FG(M')\longrightarrow M'$ by
$$a'\o(m'\o c)\mapsto \varepsilon(c)\a^{-2}_{A'}(a')\cdot \a^{-3}_{M'}(m'),$$

For all $a\in A$ and $m\in M$,
$$\begin{aligned}
\eta_M(a\cdot m)=&(1\o \a_M(a_{(0)}\cdot m_{(0)}))\o a_{(1)}\cdot m_{(1)}\\
                =&\b(\a^{2}(a_{(0)}))\o\a^2_M(m_{(0)})\o a_{(1)}\cdot m_{(1)}\\
                =&a\cdot\eta_M(m),
\end{aligned}$$
and
$$\begin{aligned}
\r(\eta_M(m))=&(1\o \a^2_M(m_{(0)}))\o m_{(1)1}\o\a^{-1}_C(m_{(1)2})\\
             =&(1\o \a_M(m_{(0)(0)}))\o m_{(0)(1)}\o m_{(1)}\\
             =&\eta_M(m_{(0)})\o m_{(1)}.
\end{aligned}$$
Hence $\eta_M$ is a morphism in $_A\mathcal{M}(H)^C$. Similarly we can check that $\delta_{M'}$ is a morphism in $_{A'}\mathcal{M}(H)^{C'}$. For all $m'\in M,c\in C$,
$$\begin{aligned}
&G(\delta_{M'})\circ\eta_{G(M')}(m'\o c)\\
&=G(\delta_{M'})(1\o\a^2_{M'}(m')\o\a_C(c_1)\o\a^{-1}_C(c_1))\\
&=m'\o c,
\end{aligned}$$
and for all $a'\in A',m\in M$,
$$\begin{aligned}
&\delta_{F(M)}\circ F(\eta_{M})(a'\o m)\\
&=\delta_{F(M)}(a'\o((1\o \a_M(m_{(0)}))\o m_{(1)}))\\
&=\a^{-2}_{A'}(a')\cdot(1\o\a^{-2}_M(m_{(0)}))\varepsilon(m_{(1)})\\
&=a'\o m.
\end{aligned}$$

The proof is completed.
\end{proof}

In the rest of this section, we will establish tensor identities involving the pair of adjoint functors $(F,G)$.

\begin{theorem}
Let $(id,id,\g):(H,A,C)\longrightarrow (H,A,C')$ be a monoidal morphism of Doi-Hopf data. Then $G(C')=C$.

Furthermore let $M\in\  _A\mathcal{M}(H)^C$ and $N\in\  _A\mathcal{M}(H)^{C'}$ be arbitrary. If $(C,\a_C)$ is a Hom-Hopf algebra, we have the following isomorphism in $_A\mathcal{M}(H)^C$
\begin{eqnarray}
&&M\o G(N)\simeq G(F(M)\o N).
\end{eqnarray}

If $(C,\a_C)$ has a twisted antipode, then
\begin{eqnarray}
&&G(N)\o M\simeq G(N\o F(M)).
\end{eqnarray}
\end{theorem}

\begin{proof}
First of all, define linear map $f:C'\square_{C'}C\rightarrow C,$ by $f(c'\o c)=\varepsilon(c')c$. For all $a\in A$,
$$\begin{aligned}
f(a\cdot(c'\o c))=&f(a_{(0)}\cdot c'\o a_{(1)}\cdot c)\\
                 =&f(\varepsilon(a_{(0)(0)})a_{(0)(1)}\cdot c'\o a_{(1)}\cdot c)\\
                 =&\varepsilon(a_{(0)})\varepsilon(c')a_{(1)}\cdot c\\
                 =&\varepsilon(c')a\cdot c=a\cdot f(c'\o c),
\end{aligned}$$
that is, $f$ is $A$-linear. And
$$\begin{aligned}
f(c'\o c)_{(0)}\o f(c'\o c)_{(1)}=&\varepsilon(c')c_1\o\a^{-1}_C(c_2)\\
                                 =&f((c'\o c)_{(0)})\o (c'\o c)_{(1)},
\end{aligned}$$
that is, $f$ is $C$-colinear.

Define $g:C\rightarrow C'\square_{C'}C$ by $g(c)=\g\a^{-1}_C(c_1)\o\a^{-1}_C(c_2)$. Easy to see that $g(c)\in C'\square_{C'}C$, and for all $c\in C,c'\in C'$,
$$\begin{aligned}
g(f(c'\o c))&=g(\varepsilon(c')c)=\varepsilon(c')\g\a^{-1}_C(c_1)\o\a^{-1}_C(c_2)\\
            &=\varepsilon(c'_1)\a^{-1}_C(c'_2)\o c=c'\o c,
\end{aligned}$$
and
$$\begin{aligned}
f(g(c))=f(\g\a^{-1}_C(c_1)\o\a^{-1}_C(c_2))=c.
\end{aligned}$$
Thus $f$ is bijective, and $G(C')=C$.

For any object $M\in\  _A\mathcal{M}(H)^C$ and $N\in\  _A\mathcal{M}(H)^{C'}$, define $\Phi_1:M\o(N\square_{C'}C)\longrightarrow (M\o N)\square_{C'}C$ by
$$m\o(n\o c)\mapsto(m_{(0)}\o n)\o\a^{-2}_C(m_{(1)}c),$$
for all $m\in M,n\in N$ and $c\in C$. Easy to see $\Phi_1$ is well defined. We need to show that $\Phi_1$ is a morphism in $M\in\  _A\mathcal{M}(H)^C$. In fact for all $a\in A$,
$$\begin{aligned}
\Phi_1(a\cdot (m\o(n\o c)))&=\Phi_1(a_1\cdot m\o(a_{2(0)}\cdot n\o a_{2(1)}\cdot c))\\
                           &=(a_{1(0)}\cdot m_{(0)}\o a_{2(0)}\cdot n)\o\a^{-2}_C((a_{1(1)}\cdot m_{(1)})(a_{2(1)}\cdot c))\\
                           &\stackrel{(\ref{A})}=(a_{(0)1}\cdot m_{(0)}\o a_{(0)2}\cdot n)\o a_{(1)}\cdot\a^{-2}_C(m_{(1)}c)\\
                           &=a_{(0)}\cdot (m_{(0)}\o n)\o a_{(1)}\cdot\a^{-2}_C(m_{(1)}c)\\
                           &=a\cdot\Phi_1(m\o(n\o c)),
\end{aligned}$$
and
$$\begin{aligned}
\r(\Phi_1(m\o(n\o c)))=&(\a_M(m_{(0)})\o\a_N(n))\o\a^{-2}_C(m_{(1)1}c_1)\o\a^{-3}_C(m_{(1)2}c_2)\\
                      =&(m_{(0)(0)}\o\a_N(n))\o\a^{-2}_C(m_{(0)(1)}c_1)\o\a^{-2}_C(m_{(1)}\a^{-1}_C(c_2))\\
                      =&\Phi_1(m_{(0)}\o(\a_N(n)\o c_1))\o\a^{-2}_C(m_{(1)}\a^{-1}_C(c_2))\\
                      =&\Phi_1((m\o(n\o c))_{(0)})\o(m\o(n\o c))_{(1)},
\end{aligned}$$
as required. Finally we show that $\Phi_1$ is bijective. Define the linear map $\Psi_1:(M\o N)\square_{C'}C\longrightarrow M\o(N\square_{C'}C)$ by
$$(m\o n)\o c \mapsto \a^{-2}_M(m_{(0)})\o(n\o S_C(\a^{-2}_C(m_{(1)}))c),$$
for all $m\in M,n\in N$ and $c\in C$.

Before proceeding we have to verify that $\Psi_1$ is well defined. For $(m\o n)\o c\in(M\o N)\square_{C'}C$, we have the relation
\begin{eqnarray}
&&m_{(0)}\o n_{(0)}\o\a^{-2}_C(\g(m_{(1)})n_{(1)})\o \a_C(c)=\a_M(m)\o \a_N(n)\o \g(c_1)\o c_2.
\end{eqnarray}

Applying $(id\o\g\o id)\circ(id\o\Delta_C\circ S_C)\circ\r_M\circ\a^{-3}_M$ to the above identity, we obtain that
$$\begin{aligned}
&\a^{-3}_M(m_{(0)(0)})\o\g S_C\a^{-3}_C(m_{(0)(1)2})\o S_C\a^{-3}_C(m_{(0)(1)1})\o n_{(0)}\o\a^{-2}_C(\g(m_{(1)})n_{(1)})\o \a_C(c)\\
&=\a^{-2}_M(m_{(0)})\o\g S_C\a^{-2}_C(m_{(1)2})\o S_C\a^{-2}_C(m_{(1)1})\o\a_N(n)\o \g(c_1)\o c_2.
\end{aligned}$$

Multiplying the second and the fifth factor, also the third and sixth factor, we obtain that
$$\begin{aligned}
&\a^{-3}_M(m_{(0)(0)})\o n_{(0)}\o\g S_C\a^{-3}_C(m_{(0)(1)2})\a^{-2}_C(\g(m_{(1)})n_{(1)})\o S_C\a^{-3}_C(m_{(0)(1)1})\a_C(c)\\
&=\a^{-2}_M(m_{(0)})\o\a_N(n)\o \g S_C\a^{-2}_C(m_{(1)2})\g(c_1)\o S_C\a^{-2}_C(m_{(1)1})c_2.
\end{aligned}$$

Thus
$$\begin{aligned}
&\a^{-2}_M(m_{(0)})\o n_{(0)}\o n_{(1)}\o S_C\a^{-1}_C(m_{(1)})\a_C(c)\\
&=\a^{-2}_M(m_{(0)})\o\a_N(n)\o \g S_C\a^{-2}_C(m_{(1)2})\g(c_1)\o S_C\a^{-2}_C(m_{(1)1})c_2,
\end{aligned}$$
which means that $\Psi_1((m\o n)\o c)\in M\o(N\square_{C'}C)$. Hence it is well defined.
Now let us check that $\Psi_1$ and $\Phi_1$ are mutual inverses.
$$\begin{aligned}
\Psi_1\circ\Phi_1(m\o(n\o c))=&\Psi_1((m_{(0)}\o n)\o\a^{-2}_C(m_{(1)}c))\\
                             =&\a^{-2}_M(m_{(0)(0)})\o(n\o S_C(\a^{-2}_C(m_{(0)(1)}))(\a^{-2}_C(m_{(1)}c))\\
                             =&\a^{-1}_M(m_{(0)})\o(n\o S_C(\a^{-2}_C(m_{(1)1}))(\a^{-3}_C(m_{(1)2})\a^{-2}_C(c))\\
                             =&m\o (n\o c),
\end{aligned}$$
and
$$\begin{aligned}
\Phi_1\circ\Psi_1((m\o n)\o c)=&\Psi_1(\a^{-2}_M(m_{(0)})\o(n\o S_C(\a^{-2}_C(m_{(1)}))c))\\
                             =&\a^{-2}_M(m_{(0)(0)})\o n\o \a^{-2}_C(\a^{-2}_C(m_{(0)(1)})(S_C(\a^{-2}_C(m_{(1)}))c))\\
                             =&(m\o n)\o c.
\end{aligned}$$
Hence $M\o G(N)\simeq G(F(M)\o N)$. Similarly one can prove $G(N)\o M\simeq G(N\o F(M))$.

The proof is completed.
\end{proof}

\begin{corollary}
Let $(H,A,C)$ be a monoidal Doi-Hopf datum, and $T:\ _A\mathcal{M}(H)^C\longrightarrow\ _HM$ a functor forgetting the right $C$-coaction. Then for any Doi-Hopf module $M$ we have the isomorphism in $_A\mathcal{M}(H)^C$:
$$M\o C\simeq T(M)\o C.$$
\end{corollary}

\begin{proof}
We have the following monoidal morphism
$$(id,id,\varepsilon):(H,A,C)\rightarrow(H,A,k).$$
Obviously $_AM(H)^k\simeq\  _AM$, and $F:_A\mathcal{M}(H)^C\simeq\  _AM$ is the functor forgetting the $C$-coaction. Hence
$$\begin{aligned}
M\o C\simeq M\o G(k)&\simeq G(F(M)\o k)\\
&=GF(M)=F(M)\o C=T(M)\o C.
\end{aligned}$$

The proof is completed.
\end{proof}

We have the dual version of Theorem 2.8.

\begin{theorem}
Let $(id,\la,id):(H,A,C)\longrightarrow (H,A',C)$ be a monoidal morphism of Doi-Hopf data. Then $F(A)=A'$.

For all $M\in\ _A\mathcal{M}(H)^C$ and $N\in\ _{A'}M(H)^{C}$, suppose that $(A,\a_A)$ is a Hom-Hopf algebra, we have the following isomorphism in $M\in\  _{A'}\mathcal{M}(H)^C$
\begin{eqnarray}
&&F(M)\o N\simeq F(M\o G(N)).
\end{eqnarray}

If $(A',\a_{A'})$ has a twisted antipode,
\begin{eqnarray}
&&N\o F(M)\simeq F(G(N)\o M).
\end{eqnarray}
\end{theorem}

\begin{proof}
First of all remind that $(A,\a_{A})$ is a left module over itself under the action $a\cdot b=\a^{-1}_{A}(a)b$ for all $a,b\in A$.

Define a linear map $f:A'\o_A A\rightarrow A$ by $a'\o a\mapsto \a^{-3}_{A'}(a')\b(\a^{-2}_{A}(a))$. It is a routine exercise to check that $f$ is well defined and a morphism in $_{A'}M(H)^{C}$. Its inverse is given by $f^{-1}:A'\rightarrow A'\o_A A, a'\mapsto \a^{2}_{A'}(a')\o 1.$ The first assertion is proved.

If $M\in\ _A\mathcal{M}(H)^C$ and $N\in\ _{A'}M(H)^{C}$, and $(A,\a_A)$ is a Hom-Hopf algebra, define
$\Phi_2:A'\o_A(M\o G(N))\longrightarrow (A'\o_A M)\o N$ by
$$a'\o(m\o n)\mapsto (a'_1\o m)\o\a^{-2}_N(a'_2\cdot n),$$
for all $a'\in A',m\in M$ and $n\in N$. We claim that $\Phi_2$ is well defined and  a morphism in $_{A'}M(H)^{C}$. Indeed for all $a\in A$,
$$\begin{aligned}
\Phi_2(a'\b(a)\o(\a_M(m)\o \a_N(n)))=&(a'_1\b(a_1)\o \a_M(m))\o\a^{-2}_N(a'_2\b(a_2)\cdot \a_N(n))\\
                                            =&(\a_{A'}(a'_1)\o a_1\cdot m)\o\a^{-2}_N(a'_2\b(a_2)\cdot \a_N(n))\\
                                            =&\Phi_2(\a_{A'}(a')\o(a\cdot(m\o n)))
\end{aligned}$$

Then for all $b'\in A'$,
$$\begin{aligned}
\Phi_2(a'\cdot(b'\o(m\o n)))=&\Phi_2(\a_{A'}(a')b'\o(\a_M(m)\o \a_N(n)))\\
                            =&(\a_{A'}(a'_1)b'_1\o\a_M(m))\o\a^{-2}_N(\a_{A'}(a'_2)b'_2\cdot \a_N(n))\\
                            =&a'_1\cdot(b'_1\o m)\o a'_2\cdot\a^{-2}_{N}(b'_2\cdot n)\\
                            =&a'\cdot\Phi_2(b'\o(m\o n)),
\end{aligned}$$
and
$$\begin{aligned}
&\Phi_2(a'\o(m\o n))_{(0)}\o\Phi_2(a'\o(m\o n))_{(1)}\\
&=(a'_1\o m)_{(0)}\o\a^{-2}_N(a'_2\cdot n)_{(0)}\o\a^{-2}_C((a'_1\o m)_{(1)}\cdot\a^{-2}_N(a'_2\cdot n)_{(1)})\\
&=(a'_{1(0)}\o m_{(0)})\o\a^{-2}_N(a'_{2(0)}\cdot n_{(0)})\o \a^{-4}_C(a'_{1(1)}\cdot m_{(1)})\cdot\a^{-4}_C(a'_{2(1)}\cdot n_{(1)})\\
&=(a'_{(0)1}\o m_{(0)})\o\a^{-2}_N(a'_{(0)2}\cdot n_{(0)})\o \a^{-2}_C(a'_{1})\cdot\a^{-4}_C(m_{(1)}n_{(1)})\\
&=\Phi_2((a'\o(m\o n))_{(0)})\o (a'\o(m\o n))_{(1)},
\end{aligned}$$
as needed.

For the inverse of $\Phi_2$, define the linear map
$\Psi_2:(A'\o_A M)\o N\longrightarrow A'\o_A(M\o G(N))$ by
$$(a'\o m)\o n\mapsto \a^{-2}_{A'}(a'_1)\o(m\o S\a^{-2}_{A'}(a'_2)\cdot n),$$
for all $a'\in A',m\in M$ and $n\in N$.

$\Psi_2$ is well defined since for all $a\in A$,
$$\begin{aligned}
&\Psi_2((a'\b(a)\o \a_M(m))\o n)\\
&=\a^{-2}_{A'}(a'_1\b(a_1))\o(\a_M(m)\o S\a^{-2}_{A'}(a'_2\b(a_2))\cdot n)\\
&=\a^{-1}_{A'}(a'_1)\o\a^{-2}_{A}(a_1)\cdot(m\o S\a^{-3}_{A'}(a'_2\b(a_2))\cdot \a^{-1}_N(n))\\
&=\a^{-1}_{A'}(a'_1)\o(\a^{-2}_{A}(a_{11})\cdot m\o \b(\a^{-3}_{A})(a_{12})[S\a^{-3}_{A'}(a'_2\b(a_2))]\cdot n)\\
&=\a^{-1}_{A'}(a'_1)\o(\a^{-1}_{A}(a_{1})\cdot m\o [\b(\a^{-4}_{A}(a_{21}))S\a^{-4}_{A'}(\b(a_{22}))]S\a^{-3}_{A'}(a'_2)\cdot n)\\
&=\a^{-1}_{A'}(a'_1)\o(a\cdot m\o S\a^{-2}_{A'}(a'_2)\cdot n)\\
&=\Psi_2((\a_{A'}(a')\o a\cdot m)\o n).
\end{aligned}$$

Then
$$\begin{aligned}
\Phi_2\Psi_2((a'\o  m)\o n)=&\Phi_2(\a^{-2}_{A'}(a'_1)\o(m\o S\a^{-2}_{A'}(a'_2)\cdot n))\\
                           =&\a^{-2}_{A'}(a'_{11})\o m\o\a^{-2}_N(\a^{-2}_{A'}(a'_{12})\cdot(S\a^{-2}_{A'}(a'_2)\cdot n))\\
                           =&\a^{-1}_{A'}(a'_{1})\o m\o\a^{-2}_N(\a^{-3}_{A'}(a'_{12})S\a^{-3}_{A'}(a'_{22})\cdot \a_N(n))\\
                           =&(a'\o m)\o n,
\end{aligned}$$
and
$$\begin{aligned}
\Psi_2\Phi_2(a'\o (m\o n))=&\Psi_2((a'_1\o m)\o\a^{-2}_N(a'_2\cdot n))\\
                          =&\a^{-2}_{A'}(a'_{11})\o m\o S\a^{-2}_N(a'_{12})\cdot\a^{-2}_N(a'_2\cdot n)\\
                          =&a'\o (m\o n).
\end{aligned}$$
If $(A',\a_{A'})$ has a twisted antipode, similarly one can prove that
$N\o F(M)\simeq F(G(N)\o M)$.

The proof is completed.

\end{proof}

\section{Braided structure on the category $_A\mathcal{M}^C$}
\def\theequation{3.\arabic{equation}}
\setcounter{equation} {0} \hskip\parindent

Let $(H,A,C)$ be a monoidal Doi-Hopf datum. In this section, we will discuss how to make the category $_A\mathcal{M}(H)^C$ to be braided.

 Consider a map $Q:C\o C\rightarrow A\o A$ with twisted convolution inverse $R$. So we have
\begin{eqnarray}
&&Q\circ\a^{\o2}_C=\a^{\o2}_A\circ Q,\\
&&Q(c_2,d_2)R(c_1,d_1)=\varepsilon(c)\varepsilon(d)1\o1,
\end{eqnarray}
for all $c,d\in C$.

When necessary we will introduce the following notation: for all $c,d\in C$,
$$Q(c,d)=q(c,d)=Q^1(c,d)\o Q^2(c,d)\in A\o A.$$

Consider two Doi-Hopf modules $M$ and $N$, define $t_{_{M,N}}:M\o N\rightarrow N\o M$ by
\begin{eqnarray}
&&t_{_{M,N}}(m\o n)=Q(\a^{-2}_C(n_{(1)}),\a^{-2}_C(m_{(1)}))(\a^{-2}_N(n_{(0)})\o \a^{-2}_M(m_{(0)})),
\end{eqnarray}
for all $m\in M$ and $n\in N$. Since $Q$ is convolution invertible, $t_{_{M,N}}$ is a bijective.

\begin{lemma}
Let $(H,A,C)$ be a Doi-Hopf datum. Then $A\o C$ becomes a Doi-Hopf module with the structure
\begin{eqnarray*}
&&a\cdot(b\o c)=\a_A(a)b\o\a_C(c),\\
&&\r(b\o c)=b_{(0)}\o c_1\o\a_C^{-2}(b_{(1)}\cdot c_2),
\end{eqnarray*}
for all $a,b\in A$ and $c\in C$.
\end{lemma}

\begin{proof}
Straightforward.
\end{proof}

\begin{lemma}
With notations as above, the linear map $t_{_{M,N}}$ is $A$-linear for all objects $(M,\a_M),(N,\a_N)$ in $_A\mathcal{M}(H)^C$ if and only if
\begin{eqnarray}
Q(a_{2(1)}\cdot d,~~a_{1(1)}\cdot c)(\a_A(a_{2(0)})\o\a_A(a_{1(0)}))=\Delta(\a^2_A(a))Q(\a_C(d),~\a_C(c)),\label{B}
\end{eqnarray}
for all $a,b\in A$ and $c,d\in C$.
\end{lemma}

\begin{proof}
Suppose that $t_{_{A\o C,A\o C}}$ is $A$-linear. For convenience we apply the switch map twice on $(A\o C)\o(A\o C)$, and consider $t_{_{A\o C,A\o C}}$ as the map $t:A\o A\o C\o C\longrightarrow A\o A\o C\o C$ given by
$$\begin{aligned}
t(a\o b&\o c\o d)\\
               &=Q(\a^{-3}_C(b_{(1)}\cdot d_1),~~\a^{-3}_C(a_{(1)}\cdot c_1))(\a^{-2}_A(b_{(0)})\o\a^{-2}_A(a_{(0)}))\o\a^{-1}_C(d_1)\o\a^{-1}_C(c_1),
\end{aligned}$$
for all $a,b\in A$ and $c,d\in C$. $t$ is $A$-linear, so
$$\begin{aligned}
&t(a\cdot(1\o 1\o c\o d))\\
&=t(\a^{2}_A(a_1)\o\a^{2}_A(a_2)\o \a_C(c)\o \a_C(d))\\
&=Q(\a^{-1}_{H}(a_{2(1)})\cdot\a^{-2}_C(d_2),~~\a^{-1}_{H}(a_{1(1)})\cdot\a^{-2}_C(c_2))(a_{2(0)}\o a_{1(0)})\o d_1\o c_1,\\
&=a\cdot t(1\o 1\o c\o d)\\
&=\Delta_A(\a_A(a))Q(\a_C(d_2)\o\a_C(c_2))\o d_1\o c_1.
\end{aligned}$$
Applying $id\o id\o\varepsilon\o\varepsilon$ to both sides, we obtain (\ref{B}).

Conversely suppose that (\ref{B}) holds, and consider Doi-Hopf modules $M$ and $N$. For all $a\in A,m\in M$ and $n\in N$,
$$\begin{aligned}
&t_{_{M,N}}(a\cdot(m\o n))=t_{_{M,N}}(a_1\cdot m\o a_2\cdot n)\\
=&Q(\a^{-2}_C(a_{2(1)}\cdot n_{(1)}),~\a^{-2}_C(a_{1(1)}\cdot m_{(1)}))(\a^{-2}_N(a_{2(0)}\cdot n_{(0)})\o \a^{-2}_M(a_{1(0)}\cdot m_{(0)}))\\
=&Q(\a^{-3}_C(a_{2(1)}\cdot n_{(1)}),~\a^{-3}_C(a_{1(1)}\cdot m_{(1)}))(\a^{-2}_A(a_{2(0)})\o\a^{-2}_A(a_{1(0)}))\cdot (\a^{-1}_N(n_{(0)})\o \a^{-1}_M(m_{(0)}))\\
=&\Delta_A(\a^{-1}_A(a))Q(\a^{-2}_C(n_{(1)}),~\a^{-2}_C(m_{(1)}))\cdot(\a^{-1}_N(n_{(0)})\o \a^{-1}_M(m_{(0)}))\\
=&a\cdot t_{_{M,N}}(m\o n).
\end{aligned}$$

The proof is completed.

\end{proof}

\begin{example}
In the above lemma, let $C=k$ and $Q=Q(1,1)$. If $t_{_{M,N}}$ is $A$-linear, then $Q\Delta^{op}(a)=\Delta(a)Q$, which means that $A$ is quasi-cocommutative.
\end{example}

\begin{lemma}
With notations as above, the linear map $t_{_{M,N}}$ is $C$-linear for all objects $(M,\a_M),(N,\a_N)$ in $_A\mathcal{M}(H)^C$ if and only if
\begin{eqnarray}
&&Q(d_1\o c_1)\o\a^{-3}(c_2d_2)\nonumber\\
&&=Q(\a^{-1}_C(d_2),~\a^{-1}_C(e_2))_{(0)}\o\a^{-4}_C\circ m_C(Q(\a^{-1}_C(d_2),~\a^{-1}_C(e_2))_{(1)}\cdot(d_1\o c_1)),\label{D}
\end{eqnarray}
for all $c,d\in C$. The $m_C$ is the multiplication map of $C$. For $a\o b\in A\o A$, we denote
$$(a\o b)_{(0)}\o(a\o b)_{(1)}=a_{(0)}\o b_{(0)}\o a_{(1)}\o b_{(1)}\in A\o A\o H\o H.$$
\end{lemma}

\begin{proof}
Suppose that $t$ (defined in the proof of Lemma 3.2) is $C$-colinear, then for all $c,d\in C$,
$$\begin{aligned}
&t(1\o1\o c\o d)_{(0)}\o t(1\o1\o c\o d)_{(1)}\\
&=Q^1(\a^{-1}_C(d_2)\o\a^{-1}_C(c_2))_{(0)}\o Q^2(\a^{-1}_C(d_2)\o\a^{-1}_C(c_2))_{(0)}\o\a^{-1}_C(d_{11})\o\a^{-1}_C(c_{11})\\
&\o\a^{-4}_C\circ m_C((Q^1(\a^{-1}_C(d_2)\o\a^{-1}_C(c_2))_{(1)}\o Q^2(\a^{-1}_C(d_2)\o\a^{-1}_C(c_2))_{(1)})\cdot(\a^{-1}_C(d_{12})\o\a^{-1}_C(c_{12})))\\
&=t((1\o1\o c\o d)_{(0)})\o (1\o1\o c\o d)_{(1)}\\
&=t(1\o1\o c_1\o d_1)\o\a^{-3}_C(c_2d_2)\\
&=Q(\a^{-1}_C(d_{12}),~\a^{-1}_C(c_{12}))\o a^{-1}_C(d_{11})\o\a^{-1}_C(c_{11})\o\a^{-3}_C(c_2d_2).
\end{aligned}$$
Applying $id\o id\o\varepsilon\o\varepsilon\o id$ to both sides, we obtain (\ref{D}).

Conversely suppose that (\ref{D}) holds, and consider Doi-Hopf modules $M$ and $N$. For all $m\in M$ and $n\in N$,
$$\begin{aligned}
&t_{_{M,N}}(m\o n)_{(0)}\o t_{_{M,N}}(m\o n)_{(1)}\\
&=Q^1(\a^{-2}_C(n_{(1)}),\a^{-2}_C(m_{(1)}))_{(0)}\cdot\a^{-2}_N(n_{(0)(0)})
\o Q^2(\a^{-2}_C(n_{(1)}),\a^{-2}_C(m_{(1)}))_{(0)}\cdot\a^{-2}_M(m_{(0)(0)})\\
&\o\a^{-2}_C\circ m_C[(Q^1(\a^{-2}_C(n_{(1)}),\a^{-2}_C(m_{(1)}))_{(1)}\\
&~~~~\o Q^2(\a^{-2}_C(n_{(1)}),\a^{-2}_C(m_{(1)}))_{(1)})\cdot(a^{-2}_C(n_{(0)(1)})\o\a^{-2}_C(m_{(0)(1)}))]\\
&=Q^1(\a^{-3}_C(n_{(1)2}),\a^{-3}_C(m_{(1)2}))_{(0)}\cdot\a^{-1}_N(n_{(0)})
\o Q^2(\a^{-3}_C(n_{(1)2}),\a^{-3}_C(m_{(1)2}))_{(0)}\a^{-1}_M(m_{(0)})\\
&\o\a^{-2}_C\circ m_C[(Q^1(\a^{-3}_C(n_{(1)2}),\a^{-3}_C(m_{(1)2}))_{(1)}\\
&~~~~\o Q^2(\a^{-3}_C(n_{(1)2}),\a^{-3}_C(m_{(1)2}))_{(1)})\cdot(a^{-2}_C(n_{(1)1})\o\a^{-2}_C(m_{(1)1}))]\\
&\stackrel{(\ref{D})}=Q(\a^{-2}_C(n_{(1)1}),~\a^{-2}_C(m_{(1)1}))\cdot(\a^{-1}_N(n_{(0)})\o\a^{-1}_M(m_{(0)}))\o\a^{-3}_C(m_{(1)2}n_{(1)2})\\
&=t_{_{M,N}}(m_{(0)}\o n_{(0)})\o (m\o n)_{(1)}.
\end{aligned}$$

The proof is completed.

\end{proof}

\begin{example}
In the above lemma, suppose that $A=k$. If $t_{_{M,N}}$ is $C$-colinear, then
$Q(d_2,c_2)d_1c_1=Q(d_1,c_1)c_2d_2$, which means that $C$ is quasi-commutative.
\end{example}

\begin{lemma}
With notations as above, the relation $(1.1)$ holds for all objects $M,N,P$ in $_A\mathcal{M}(H)^C$ if and only if
\begin{eqnarray}
&&(\Delta\o\a_A)Q(\a^{-2}_C(d)\a^{-1}_C(e),~\a^{-1}_C(c))\nonumber\\
&&=\a^2_AQ^1(\a^{-2}_C(d),\a^{-3}_C(c_2))\nonumber\\
&&\quad \o q(e,\a^{-1}_H(Q^2(\a^{-2}_C(d),\a^{-3}_C(c_2))_{(1)})\cdot\a^{-3}_C(c_1))(1\o Q^2(\a^{-2}_C(d),\a^{-3}_C(c_2))_{(0)}),\label{C}
\end{eqnarray}
for all $a,b\in A$ and $c,d\in C$.
\end{lemma}

\begin{proof}
Let $M,N,P$ be Doi-Hopf modules. For all $m\in M,n\in N$ and $p\in P$, we have the left side of (1.1):
$$\begin{aligned}
&a_{_{N,P,M}}t_{_{M,N\o P}}a_{_{M,N,P}}((m\o n)\o p)\\
&=a_{_{N,P,M}}[ Q(\a^{-2}_C((n\o\a_P(p))_{(1)}),~~\a^{-3}_C(m_{(1)}))\cdot\\
&\quad\cdot((\a^{-2}_N\o\a^{-2}_P)((n\o\a_P(p))_{(0)})\o\a^{-3}_M(m_{(0)}))]\\
&=a_{_{N,P,M}}[(\Delta_A\o id)Q(\a^{-4}_C(n_{(1)}\a_P(p_{(1)})),~~\a^{-3}_C(m_{(1)}))\cdot\\
&\quad\cdot(\a^{-2}_N(n_{(0)})\o\a^{-1}_P(p_{(0)})\o\a^{-3}_M(m_{(0)}))]\\
&=(\a^{-1}_A\o id\o id)(\Delta_A\o\a_A)Q(\a^{-4}_C(n_{(1)}\a_C(p_{(1)})),~~\a^{-3}_C(m_{(1)}))\cdot\\
&\quad\cdot(\a^{-3}_N(n_{(0)})\o\a^{-1}_P(p_{(0)})\o\a^{-2}_M(m_{(0)})),
\end{aligned}$$

and the right side of (1.1):
$$\begin{aligned}
&(id_N\o t_{_{M,P}})a_{_{N,M,P}}(t_{_{M,N}}\o id_P)((m\o n)\o p)\\
&=(id_N\o t_{_{M,P}})[\a^{-3}_N(Q^1(n_{(1)},m_{(1)})\cdot n_{(0)})\o \a^{-2}_A (Q^2(n_{(1)},m_{(1)}))\cdot\a^{-2}_M(m_{(0)}))\o\a_P(p)]\\
&=\a^{-3}_N(Q^1(n_{(1)},m_{(1)})\cdot n_{(0)})\o q(\a^{-1}_C(p_{(1)}),~~\a^{-4}_H(Q^2(n_{(1)},m_{(1)})_{(1)})\cdot\a^{-4}_C(m_{(0)(1)}))\cdot\\
&\cdot(\a^{-1}_P(p_{(0)})\o \a^{-4}_M(Q^2(n_{(1)},m_{(1)})_{(0)}\cdot m_{(0)(0)})).
\end{aligned}$$
Clearly if $(\ref{C})$ holds, we obtain the relation $(1.1)$.

Conversely suppose that the relation $(1.1)$ holds in the category $_A\mathcal{M}(H)^C$. Taking $U=V=W=A\o C$ and applying the switch map as in the previous lemma, we regard both sides of $(1.1)$ as maps
$$A^{\o3}\o C^{\o3}\longrightarrow A^{\o3}\o C^{\o3}.$$

On one hand we have that
$$\begin{aligned}
&a_{_{V,W,U}}t_{_{U,V\o W}}a_{_{U,V,W}}(1\o1\o1\o c\o d\o e)\\
&=(\a_A^{-1}\o id\o id)(\Delta\o\a_A)Q(\a^{-3}_C(d_2)\a^{-2}_C(e_2),\a^{-2}_C(c_2))\o\a^{-1}_C(d_1)\o e_1\o\a^{-1}_C(c_1),
\end{aligned}$$

and on the other hand
$$\begin{aligned}
&(id_V\o t_{_{U,W}})a_{_{V,U,W}}(t_{_{U,V}}\o id_W)(1\o1\o1\o c\o d\o e)\\
&=\a_A(Q^1(\a^{-3}_C(d_2),~\a^{-3}_C(c_2)))\\
&\o q(\a^{-1}_C(e_2),~\a^{-1}_H(Q^2(\a^{-3}_C(d_2),~\a^{-3}_C(c_2))_{(1)})\cdot\a^{-4}_C(c_{12}))(1\o Q^2(\a^{-3}_C(d_2),~\a^{-3}_C(c_2))_{(0)})\\
&\o\a^{-1}_C(d_1)\o e_1\o\a^{-2}_C(c_{11}).
\end{aligned}$$
Applying $id^{\o3}\o\varepsilon^{\o3}_C$ to both sides, we obtain (\ref{C}).

The proof is completed.

\end{proof}

\begin{lemma}
With notations as above, the relation $(1.2)$ holds for all objects $M,N,P$ in $_A\mathcal{M}(H)^C$ if and only if
\begin{eqnarray}
&&(\a_A\o\Delta_A)Q(\a^{-1}_C(e),~\a^{-1}_C(c)\a^{-2}_C(d))\nonumber\\
&&=q(\a^{-1}_H(Q^1(\a^{-3}_C(e_2),~\a^{-2}_C(d))_{(1)})\cdot\a^{-3}_C(e_1),~~c)( Q^1(\a^{-3}_C(e_2),~\a^{-2}_C(d))_{(0)}\o1)\nonumber\\
&&\o\a^2_A(Q^2(\a^{-3}_C(e_2),~\a^{-2}_C(d)))\label{E},
\end{eqnarray}
for all $c,d,e\in C$.
\end{lemma}

\begin{proof}
The proof is similar to the proof of the above lemma, and left to the reader.
\end{proof}

We summarize our results as follows.

\begin{theorem}
Let $(H,A,C)$ be a monoidal Doi-Hopf datum, and $Q:C\o C\rightarrow A\o A$ a twisted convolution invertible map. Then the family of maps $t_{_{M,N}}:M\o N\rightarrow N\o M$ given by
\begin{eqnarray*}
&&t_{_{M,N}}(m\o n)=Q(\a^{-2}_C(n_{(1)}),\a^{-2}_C(m_{(1)}))\cdot(\a^{-2}_N(n_{(0)})\o \a^{-2}_M(m_{(0)})),
\end{eqnarray*}
defines a braiding on the category $_A\mathcal{M}(H)^C$ if and only if $Q$ satisfies the relations \ref{B}--\ref{E}.

\end{theorem}

\begin{example}
(1) If $C=k$ and denote $Q=Q(1)=Q^1\o Q^2=q^1\o q^2\in A\o A$. Then $\a_A^{\o2}(Q)=Q(1)=Q$, i.e., $Q$ is $\a_A$-invariant.

Relations (\ref{C}) and (\ref{E}) can be rewritten as
\begin{eqnarray*}
&&(\Delta_A\o\a_A)Q=\a^2_A(Q^1)\o\a_A(q^1)\o q^2\a_A(Q^1)=\a_A(Q^1)\o\a_A(q^1)\o q^2Q^2,\\
&&(\a_A\o\Delta_A)Q=q^1\a_A(Q^1)\o\a_A(q^1)\o\a^2_A(Q^2)=q^1Q^1\o\a_A(q^2)\o\a_A(Q^2).
\end{eqnarray*}
We conclude that the relations of Theorem 3.8 are satisfied if and only if $(A,\a_A,Q^{-1})$ is quasitriangular.

(2) Hom-bialgebra $(H,\a_H)$ is called coquasitriangular if there exists a linear map $\sigma:H\o H\rightarrow k$ such that for all $a,b,c\in H$
\begin{eqnarray*}
&&\sigma(b_2,a_2)b_1a_1=\sigma(b_1,a_1)a_2b_2,\\
&&\sigma(ab,c)=\sigma(\a_H(a),c_1)\sigma(\a_H(b),c_2),\\
&&\sigma(c,ab)=\sigma(c_1,\a_H(a))\sigma(c_2,\a_H(b)).
\end{eqnarray*}

 If $A=k$, relations (\ref{C}) and (\ref{E}) take the form
\begin{eqnarray*}
&&Q(de,~\a_C(c))=Q(\a_C(e),~c_1) Q(\a_C(d),c_2),\\
&&Q(\a_C(e),cd)=Q(e_1,\a_C(c))Q(\a_C(e_2),d).
\end{eqnarray*}
We conclude that the relations of Theorem 3.8 are satisfied if and only if $(C,\a_C,Q)$ is coquasitriangular.

(3) Let $(H,\a)$ be a Hom-Hopf algebra with bijective antipode. We have verified that the category of Yetter-Drinfeld modules $_H\mathcal{YD}^H\simeq\ _H\mathcal{M}(H^{op}\o H)^H$ is a monoidal category. As we all know, $_H\mathcal{YD}^H$ is braided and the braiding is given by
$$t_{_{M,N}}:M\o N\rightarrow N\o M,~~m\o n\mapsto \a^{-1}_N(n_{(0)})\o\a^{-1}_M(\a^{-1}_H(n_{(1)})\cdot m).$$
The corresponding map $Q$ is
$$Q(h\o g)=\varepsilon(g)(1\o h).$$
It is straightforward to check that $Q$ satisfies the conditions in Theorem 3.8.
\end{example}

\section{The generalized smash product of Hom-bialgebra and Drinfeld double}
\def\theequation{4.\arabic{equation}}
\setcounter{equation} {0} \hskip\parindent

In this section, we will introduce a new algebra such that its module category is isomorphic to the category $_A\mathcal{M}(H)^C$. Moreover this new algebra is more general than Drinfeld double and is quasitriangular when $_A\mathcal{M}(H)^C$ is braided.

\begin{proposition}
Let $(H,\a_H)$ be a Hom-bialgebra, $(A,\a_A)$ a right $H$-comodule algebra and $(B,\a_B)$ a left $H$-module algebra. Then we have a Hom-algebra $(A\#B,\a_A\o\a_B)$ which is equal to $A\o B$ as a vector space, with the following multiplication
$$(a\#b)(a'\#b')=a\a^{-1}_A(a'_{(0)})\#(\a^{-1}_B(b)\leftharpoonup\a^{-2}_H(a'_{(1)}))b',$$
for all $a,a'\in A$ and $b,b'\in B$. The unit is $1_A\#1_B$.
\end{proposition}

\begin{proof}
The proof is a routine computation.
\end{proof}

If $(C,\a_C)$ is a left $H$-module coalgebra, then $(C^*,(\a^{-1}_C)^*)$ is a right $H$-module algebra, and the right $H$-action is given by
$$\langle f\leftharpoonup h,c\rangle =\langle f,h\cdot\a^{-2}_C(c)\rangle ,$$
for all $f\in C^*,c\in C$ and $h\in H$.

\begin{lemma}
Let $(H,A,C)$ be a Hom-type Doi-Hopf datum. Then we have the category isomorphism $_A\mathcal{M}(H)^C\simeq\ _{A\#C^*}\mathcal{M}$.
\end{lemma}

\begin{proof}
For any object $(M,\a_M)$ in $_A\mathcal{M}(H)^C$, define the left action of $A\#C^*$ on $M$ by
$$(a\#f)\cdot m=\langle f,m_{(1)}\rangle a\cdot\a^{-1}_M(m_{(0)}),$$
for all $a\in A,f\in C^*$ and $m\in M$. Indeed for all $b\in A$ and $g\in C^*$,
$$\begin{aligned}
(a\#f)(b\#g)\cdot \a_M(m)&=(a\a^{-1}_A(b_{(0)})\#(\a^*(f)\leftharpoonup\a^{-2}_H(b_{(1)}))g)\cdot\a_M(m)\\
                         &=\langle (\a^*(f)\leftharpoonup\a^{-2}_H(b_{(1)}))g,\a_M(m_{(1)})\rangle a\a^{-1}_A(b_{(0)})\cdot m_{(0)}\\
                         &=\langle f,\a^{-1}_H(b_{(1)})\cdot\a^{-2}_M(m_{(1)1})\rangle \langle g,\a^{-1}_M(m_{(1)2})\rangle a\a^{-1}_A(b_{(0)})\cdot m_{(0)}\\
                         &=\langle g,m_{(1)}\rangle \langle f,\a^{-1}_H(b_{(1)})\cdot\a^{-2}_C(m_{(0)(1)})\rangle\\
                         &~~~~~ \a_A(a)\cdot\a^{-1}_M(b_{(0)}\cdot\a^{-1}_M(m_{(0)(0)}))\\
                         &=\langle g,m_{(1)}\rangle (\a_A(a)\#(\a^{-1})^*(f))\cdot(b\cdot\a^{-1}(m_{(0)}))\\
                         &=(\a_A(a)\#(\a^{-1})^*(f))\cdot((b\#g)\cdot m),
\end{aligned}$$
and
$(1\#\varepsilon_C)\cdot m=\a_M(m)$.

Conversely suppose that $M\in\ _{A\#C^*}\mathcal{M}$. Easy to see that $M$ is a left $A$-module by
$$a\cdot m=(a\#\varepsilon_C)\cdot m,$$
for all $a\in A$ and $m\in M$.

Also $M$ is a right $C$-comodule by
$$m_{(0)}\o m_{(1)}=\sum (1\#e^i)\cdot m\o e_i,$$
where $\{e_i\}$ and $\{e^i\}$ is dual basis of $C$.

In fact for all $f,g\in C^*$,
$$\begin{aligned}
\a_M((1\#e^i)\cdot m)\langle f, e_{i1}\rangle\langle g, e_{i2}\rangle&=(1\#(\a^{-1}_C)^*(e^i))\cdot\a_M(m)\langle f*g, \a^2_C(e_i)\rangle\\
&=(1\#(\a_C)^*(f*g))\cdot\a_M(m)\\
&=(1\#f)\cdot((1\#(\a_C)^*(e^i))\cdot m)\\
&=(1\#e^j)\cdot((1\#e^i)\cdot m)\langle f, e_{j}\rangle\langle g,\a_C(e_i)\rangle,
\end{aligned}$$
Therefore $\sum_i\a_M((1\#e^i)\cdot m)\o e_{i1}\o e_{i2}=\sum_{i,j}(1\#e^j)\cdot((1\#e^i)\cdot m)\o e_{j}\o\a_C(e_i).$
And
$$\sum_i (1\#e^i)\cdot m \varepsilon_C(e_i)=(1\#\varepsilon_C)\cdot m=\a_M(m).$$
For all $a\in A$, we have
$$\begin{aligned}
a_{(0)}\cdot m_{(0)}\o a_{(1)}\cdot m_{(1)}&=\sum_i a_{(0)}\cdot((1\#e^i)\cdot m)\o a_{(1)}\cdot e_i\\
&=\sum_i ((\a^{-1}_A(a_{(0)})\#\varepsilon_C)(1\#e^i))\cdot \a_M(m)\o a_{(1)}\cdot e_i\\
&=\sum_i (a_{(0)}\#(\a^{-1})^*(e^i))\cdot \a_M(m)\o a_{(1)}\cdot e_i\\
&=\sum_i (1\#(\a^{-1})^*(e^i)(a\#\varepsilon_C)\cdot \a_M(m)\o e_i\\
&=\sum_i (1\#e^i)\cdot(a\cdot m)\o e_i\\
&=(a\cdot m)_{(0)}\o(a\cdot m)_{(1)}.
\end{aligned}$$

The proof is completed.

\end{proof}

\begin{proposition}
Let $(H,\a_H)$ be a Hom-bialgebra, $(A,\a_A)$ a right $H$-comodule algebra and $(B,\a_B)$ a left $H$-module algebra. Assume also that $A$ and $B$ are Hom-bialgebra. Then $(A\#B,\a_A\o\a_B)$ is a Hom-bialgebra with the tensor coproduct if and only if
\begin{eqnarray}
&&\Delta_A(a_{(0)})\o b_1\leftharpoonup\a^{-1}_H(a_{(1)1})\o b_2\leftharpoonup\a^{-1}_H(a_{(1)2})\nonumber\\
~~~~~~&&=a_{1(0)}\o a_{2(0)}\o(b_1\leftharpoonup \a^{-1}_H(a_{1(1)}))\o(b_2\leftharpoonup \a^{-1}_H(a_{2(1)})),\\
&&\varepsilon_A(a_{(0)})\varepsilon_B(b\leftharpoonup\a^{-1}_H(a_{(1)}))=\varepsilon_A(a)\varepsilon_B(b).
\end{eqnarray}

Moreover if both of $A$ and $B$ are Hom-Hopf algebras, then $A\#B$ is also a Hom-Hopf algebra with the antipode
$$S(a\#b)=(1\#S_B\a^{-1}_B(b))(S_A\a^{-1}_A(a)\#1),$$
for all $a\in A$ and $b\in B$.
\end{proposition}

\begin{proof}
The proof is straightforward and left to the reader.
\end{proof}

\begin{proposition}
Let $(H,A,C)$ be a monoidal Doi-Hopf datum. Then $(A\#C^*,\a_A\o(\a^{-1}_C)^*)$ is a Hom-bialgebra and consequently $_A\mathcal{M}(H)^C$ and $_{A\#C^*}\mathcal{M}$ are isomorphic as monoidal category.
\end{proposition}

\begin{proof}
For all $a\in A$, $c,d\in C$ and $f\in C^*$,
$$\begin{aligned}
&\Delta_A(a_{(0)})\langle f_1\leftharpoonup\a^{-1}_H(a_{(1)1}),c\rangle\langle f_2\leftharpoonup\a^{-1}_H(a_{(1)2}),d\rangle\\
&=\Delta_A(a_{(0)})\langle f_1,\a^{-1}_H(a_{(1)1})\cdot\a^{-2}_C(c)\rangle\langle f_2,\a^{-1}_H(a_{(1)2})\cdot\a^{-2}_C(d)\rangle\\
&\stackrel{(\ref{A})}=a_{(0)1}\o a_{(0)2}\langle f,\a^{-1}_H(a_{(1)})\cdot\a^{-4}_C(cd)\rangle\\
&=a_{(0)1}\o a_{(0)2}\langle f\leftharpoonup\a^{-1}_H(a_{(1)}),\a^{-2}_C(cd)\rangle\\
&=a_{(0)1}\o a_{(0)2}\langle f_1\leftharpoonup\a^{-1}_H(a_{(1)1}),c\rangle\langle f_2\leftharpoonup\a^{-1}_H(a_{(1)2}),d\rangle,
\end{aligned}$$
and
$$
\varepsilon_A(a_{(0)})\langle f,\a^{-1}_H(a_{(1)1})\cdot 1_C\rangle=\varepsilon_A(a)\langle f,1_C\rangle.
$$
Thus $(4.1)$ and $(4.2)$ are satisfied.

The proof is completed.

\end{proof}

\begin{example}
Consider the monoidal Doi-Hopf datum $(H,H,H^{op}\o H)$ defined in Proposition 2.5. The algebra $A\#C^*$ is nothing else than the Drinfeld double $(H\bowtie H^*,\a_H\o(\a^{-1})^*)$. Thus we recover the multiplication, comultiplication and antipode by
\begin{eqnarray*}
&&(a\bowtie f)(b\bowtie g)=a\a^{-2}_H(b_{12})\bowtie[S^{-1}\a^{-3}_H(b_{11})\rightharpoonup((\a^*_H)^2)(f)\leftharpoonup\a^{-3}_H(b_2)]g,\\
&&\Delta(h\o f)=h_{1}\bowtie f_{1}\o h_{2}\bowtie f_{2},\\
&&S(a\bowtie f)=(1\#S^*_H\a^*_H(f))(S_H\a^{-1}_H(a)\#\varepsilon_H),
\end{eqnarray*}
for all $a,b\in H$, $f,g\in H^*$.
\end{example}

Now let $(H,A,C)$ be a monoidal Doi-Hopf datum and assume that we have a twisted convolution inverse map $Q:C\o C\rightarrow A\o A$ satisfying (\ref{B})--(\ref{E}). $Q$ induces a map
$$\widetilde{Q}:k\longrightarrow (A\#C^*)\o(A\#C^*),$$
for all $c,d\in C$ and $f,g\in A^*$.

The braiding on $_A\mathcal{M}(H)^C$ translates into a braiding on $_{A\#C^*}\mathcal{M}$. This means that $A\#C^*$ is quasitriangular and the corresponding $R\in(A\#C^*)\o(A\#C^*)$ is just $\widetilde{Q}(1)$.

In particular for a finite dimensional Hom-Hopf algebra $(H,\a)$, the Drinfeld double $D(H)$ is quasitriangular.

\section*{ACKNOWLEDGEMENTS}

This work was supported by the NSF of China (No. 11371088) and the NSF of Jiangsu Province (No. BK2012736).

 Daowei Lu, Department of Mathematics, Southeast University, Nanjing, Jiangsu Province, 210096, P. R. China, e-mail: ludaowei620@126.com.

 Shuanhong Wang, Department of Mathematics, Southeast University, Nanjing, Jiangsu Province, 210096, P. R. China.


\begin{thebibliography}{aa}

\bibitem{AGS} L. Alvarez-Gaum¨¦, C. Gomez and G. Sierra, Quantum group interpretation of some conformal field theories. Phys. Lett. B {\bf 220}(1989), 142--152.


\bibitem{AS} N. Aizawa, H. Sato, $q$-deformation of the Virasoro algebra with central extension. Phys. Lett. B {\bf 256} (1991), 185--190.


\bibitem{CMZ} S. Caenepeel, G. Militaru, S. Zhu, Crossed modules and Doi-Hopf modules. Israel J. Math. {\bf 100}(1997), 221--247.


\bibitem{CVZ} S. Caenepeel, F. Van Oystaeyen, B. Zhou, Making the category of Doi-Hopf
modules into a braided monoidal category. Algebras Representation Theory {\bf 1}(1998), 75--96.


\bibitem{D} Y. Doi, Unifying Hopf modules. J. Algebra {\bf 153} (1992), 373--385.


\bibitem{HLS} J. T. Hartwig, D. Larsson and S. D. Silvestrov, Deformations of Lie algebras using $\sigma$-derivations. J. Algebra {\bf 295} (2006), 314--361.


\bibitem{K} C. Kassel, Quantum Group. Grad. Texts in Math., vol. {\bf 155}, Springer Berlin, 1995.


\bibitem{M} S. Majid, Foundations of quantum group theory. Cambridge university press, 1995.


\bibitem{MP} A. Makhlouf, F. Panaite, Hom-L-R-smash products, Hom-diagonal crossed products and the Drinfel'd double of a Hom-Hopf algebra. J. Algebra {\bf 441}(2015), 314--343.


\bibitem{MS1} A. Makhlouf, S. D. Silvestrov, Hom-algebra structure. J. Gen. Lie Theory Appl. {\bf 2}(2008), 52--64.


\bibitem{MS2} A. Makhlouf, S. D. Silvestrov, Hom-algebras and Hom-coalgebras. J. Alg. Appl. {\bf 9}(2010), 553--589.


\bibitem{RT} D. E. Radford, J. Towber, Yetter-Drinfeld categories associated to an arbitrary
bialgebra. J. Pure Appl. Algebra {\bf 87}(1993), 259--279.


\bibitem{S} M. E. Sweedler, Hopf Algebras. New York: Benjamin, 1969.


\bibitem{Y} D. Yau, Hom-quantum groups: I. Quasi-triangular Hom-bialgebras. J. Phys. A: Math. Theor. {\bf 45}(2012), 065203.

\end{thebibliography}
\end{document}